\documentclass{amsart}
\usepackage{amsrefs}
\usepackage{graphicx}
\usepackage{amsmath,amssymb}
\usepackage[cmtip,all]{xy}
\usepackage{color}

\usepackage{amsfonts}

\newcommand{\GG}{\mathbb{G}}
\newcommand{\PP}{\mathbb{P}}
\newcommand{\QQ}{\mathbb{Q}}

\newcommand{\Aut}{\mathrm{Aut}}
\newcommand{\Ob}{\mathrm{Ob}}
\newcommand{\GL}{\mathrm{GL}}
\newcommand{\PGL}{\mathrm{PGL}}
\newcommand{\Gal}{\mathrm{Gal}}

\newcommand{\Br}{\mathrm{Br}}

\newcommand{\kbar}{\overline{k}}

\newcommand{\sC}{\mathcal{C}}
\newcommand{\sD}{\mathcal{D}}

\DeclareFontFamily{U}{wncy}{}
\DeclareFontShape{U}{wncy}{m}{n}{%
   <5>wncyr5%
   <6>wncyr6%
   <7>wncyr7%
   <8>wncyr8%
   <9>wncyr9%
   <10>wncyr10%
   <11>wncyr10%
   <12>wncyr6%
   <14>wncyr7%
   <17>wncyr8%
   <20>wncyr10%
   <25>wncyr10}{}
\DeclareSymbolFont{mcy}{U}{wncy}{m}{n}
\DeclareMathAlphabet{\cyr}{U}{wncy}{m}{n}
\DeclareMathSymbol{\Sha}{\mathord}{mcy}{"58}

\newtheorem{theorem}{Theorem}[section]
\newtheorem{lemma}[theorem]{Lemma}

\newtheorem{prop}[theorem]{Proposition}
\theoremstyle{definition}
\newtheorem{definition}[theorem]{Definition}
\newtheorem{question}[theorem]{Question}
\newtheorem{remark}[theorem]{Remark}

\numberwithin{equation}{section}
\begin{document}
\title{Visualizing elements of Sha[3] in genus 2 jacobians}
\author{Nils Bruin \and Sander R. Dahmen}
\thanks{Research of both authors supported by NSERC}
\address{Department of Mathematics,
         Simon Fraser University,
         Burnaby, BC,
         Canada}
\email{nbruin@sfu.ca}
\email{sdahmen@irmacs.sfu.ca}

\date{January 28,  2010}
\maketitle
\begin{abstract}
Mazur proved that any element $\xi$ of order three in the
Shafa\-revich-Tate group of an elliptic curve $E$ over a number
field $k$ can be made visible in an abelian surface $A$ in the sense
that $\xi$ lies in the kernel of the natural homomorphism between
the cohomology groups $H^1(\Gal(\overline{k}/k),E) \to
H^1(\Gal(\overline{k}/k),A)$. However, the abelian surface in
Mazur's construction is almost never a jacobian of a genus 2
curve. In this paper we show that any element of order three in
the Shafarevich-Tate group of an elliptic curve over a number
field can be visualized in the jacobians of a genus $2$ curve.
Moreover, we describe how to get explicit models of the genus $2$
curves involved.
\end{abstract}


\section{Introduction}

Let $E$ be an elliptic curve over a field $k$ with separable
closure $\kbar$. We write
$H^1(k,E[3]):=H^1(\Gal(\kbar/k),E[3](\kbar))$ for the first galois
cohomology group taking values in the $3$-torsion of $E$ (the
notation $H^i(k,A)$ is used similarly for other group schemes
$A/k$ later in this paper). We are primarily concerned with the
question which $\delta\in H^1(k,E[3])$ are \emph{visible} in the
jacobian of a genus $2$ curve. Mazur defines \emph{visibility} in
the following way. Let $0\to E \to A \to B\to 0$ be a short exact
sequence of abelian varieties over $k$. By taking galois
cohomology, we obtain the exact sequence
\begin{equation}\label{eq:def-visibility}
\xymatrix{A(k) \ar[r] & B(k) \ar[r] & H^1(k,E) \ar[r]^{\phi} & H^1(k,A)}.
\end{equation}
Elements of the kernel of $\phi$ are said to be \emph{visible} in
$A$. Mazur chose this term because a model of the principal
homogeneous space corresponding to an element $\xi \in H^1(k,E)$
that is visible in $A$ can be obtained as a fiber of $A$ over a
point in $B(k)$ (this can readily be seen from
(\ref{eq:def-visibility})).
By extension, we say that $\delta\in H^1(k,E[n])$ is visible in
$A$ if the image of $\delta$ under the natural homomorphism
$H^1(k,E[n])\to H^1(k,E)$ is visible in $A$.

Let us restrict to the case that $k$ is a number field for the
rest of this section. Inspired by some surprising experimental
data \cite{cremaz:vis}, Mazur \cite{mazur:sha3} proved, that for
any element $\xi$ in the Shafarevich-Tate group $\Sha(E/k)$ of
order three, there exists an abelian variety $A$ over $k$ such
that $\xi$ is visible in $A$. The abelian variety that Mazur
constructs is almost never principally polarizable over $\kbar$
and hence is almost never a jacobian of a genus $2$ curve. In the
present paper, we show that any element from $\Sha(E/k)[3]$ is in
fact visible in the jacobian of a genus $2$ curve. Moreover, we
describe how to get an explicit model of such a genus $2$ curve.

\section{Torsors and theta groups}
\label{sec:tortheta}

Throughout this section let $n>1$ be an integer, let $k$ be a
perfect field of characteristic not dividing $n$ and let $E$
denote an elliptic curve over $k$. In \cite{CFOSS:descent_alg},
many equivalent interpretations are given for the group
$H^1(k,E[n])$. For our purposes, we need two classes of objects.
The first is most closely related with descent in general and our
question in particular. We consider $E$-torsors under
$E[n](\kbar)$  and, following \cite{CFOSS:descent_alg}, call them
$n$-coverings.
\begin{definition}
An $n$-covering $\pi:C\to E$ of an elliptic curve $E$ is an
unramified covering over $k$ that is galois and irreducible over
$\kbar$, with $\Aut_{\kbar}(C/E)\simeq E[n](\kbar)$. Two
$n$-coverings $\pi_1: C_1 \to E$, $\pi_2: C_2 \to E$ are called
isomorphic if there exists a $k$-morphism $\phi:C_1 \to C_2$ such
that $\pi_1 = \pi_2 \circ \phi$.
\end{definition}
Over $\kbar$, all $n$-coverings are isomorphic to the trivial
$n$-covering, the multiplication-by-$n$ map $[n]: E\to E$.
\begin{prop}[\cite{CFOSS:descent_alg}*{Proposition~1.14}]
The $k$-isomorphism classes  of $n$-coverings of $E$ are
classified by $H^1(k,E[n])$.
\end{prop}
For $\delta \in H^1(k,E[n])$ we denote by $C_{\delta}$ the curve
in the covering $C_{\delta} \to E$ corresponding to $\delta$. We
remark that $\delta \in H^1(k,E[n])$ has trivial image in
$H^1(k,E)$ if and only if $C_{\delta}$ has a $k$-rational point.

We write $O$ for the identity on $E$. The complete linear system
$|n\cdot O|$ determines a morphism
$E\to \PP^{n-1}$, where the translation action of $E[n]$ extends
to a linear action on $\PP^{n-1}$. This gives a projective
representation $E[n]\to\PGL_n$. The lift of this representation to
$\GL_n$ gives rise to a  group $\Theta_E$, which fits in the
following diagram.
\begin{equation}\label{eq:theta-GLn}
\xymatrix{
1\ar[r]&
  \GG_m\ar[r]^{\alpha_E}\ar@{=}[d]&
  \Theta_E\ar[r]^{\beta_E}\ar[d]&
  E[n]\ar[r]\ar[d]&
  1\\
1\ar[r]&
  \GG_m\ar[r]&
  \GL_n\ar[r]&
  \PGL_n\ar[r]&
  1
}
\end{equation}
The group $E[n](\kbar)$ carries additional structure. It also has
the Weil pairing $e_E$, which is a non-degenerate alternating
galois covariant pairing taking values in the $n$-th roots of
unity
$$e_E: E[n](\kbar)\times E[n](\kbar)\to \mu_n(\kbar).$$
The commutator of $\Theta_E$ corresponds to the Weil pairing,
meaning that for $x,y\in \Theta_E$ we have
$$xyx^{-1}y^{-1}=\alpha_E(e_E(\beta_E(x),\beta_E(y))).$$

\begin{definition}\label{def:theta group}
A theta group for $E[n]$ is a central extension of group schemes
$$1\to \GG_m\stackrel{\alpha}{\to}\Theta\stackrel{\beta}{\to} E[n]\to 1$$
such that the Weil-pairing on $E[n]$ corresponds to the
commutator, i.e. for $x,y\in\Theta$ we have
$$xyx^{-1}y^{-1}=\alpha(e_E(\beta(x),\beta(y))).$$
Two theta groups
\[1 \to \GG_m \to \Theta_i \to E[n] \to 1, \qquad i=1,2\]
are called isomorphic if there exists a group scheme isomorphism
$\phi: \Theta_1 \to \Theta_2$ over $k$ making the following
diagram commutative.
\begin{equation*}
\xymatrix{ 1\ar[r]&
  \GG_m\ar[r]\ar@{=}[d]&
  \Theta_1\ar[r] \ar[d]^{\phi}&
  E[n]\ar[r]\ar@{=}[d]&
  1\\
1\ar[r]&
  \GG_m\ar[r]&
  \Theta_2\ar[r]&
  E[n]\ar[r]&
  1
}
\end{equation*}

\end{definition}
Over $\kbar$, all theta-groups are isomorphic to $\Theta_E$ as
central extensions; see \cite{CFOSS:descent_alg}*{Lemma~1.30}.

\begin{prop}[\cite{CFOSS:descent_alg}*{Proposition~1.31}]\label{prop:H1theta}
Let $E[n]$ be the $n$-torsion subscheme of an elliptic curve $E$
over a field $k$, equipped with its Weil pairing. The isomorphism
classes of theta-groups for $E[n]$ over $k$ are classified by
$H^1(k,E[n])$.
\end{prop}

The theta group associated to $\delta\in H^1(k,E[n])$ may allow
for a matrix representation $\Theta\to \GL_n$ that fits in a
diagram like \eqref{eq:theta-GLn}. This is measured by the
\emph{obstruction map} $\Ob$ introduced in \cite{oneil:PIobs} and
\cite{CFOSS:descent_alg}. This map can be obtained by taking
non-abelian galois cohomology of the defining sequence of
$\Theta_E$:
$$\cdots\longrightarrow H^1(k,\Theta_E)\longrightarrow H^1(k,E[n])\stackrel{\Ob}{\longrightarrow} H^2(k,\GG_m)=\Br(k)\longrightarrow\cdots .$$
Note that, except in some trivial cases, $\Ob$ is \emph{not} a
group homomorphism. The map $\Ob$ also has an interpretation in
terms of $n$-coverings. Let $C \to E$ be an $n$-covering
associated to $\delta \in H^1(k,E[n])$. We have that
$\Ob(\delta)=0$ if and only if $C$ admits a model $C\to\PP^{n-1}$
with $\Aut_{\kbar}(C/E)=E[3](\kbar)$ acting linearly, in which
case $C$ is $\kbar$-isomorphic to $E$ as a curve and the covering
$C\to E$ is simply a translation composed with
multiplication-by-$n$.

\begin{remark}\label{rem:brauerob}
Note that if $k$ is a number field, then any element in $\Br(k)$
that restricts to the trivial element in $\Br(k_v)$ in all
completions $k_v$ of $k$, is trivial itself. It follows that $\Ob$
is trivial on the $n$-Selmer group $S^{(n)}(E/k)$.
 \end{remark}

\section{Visibility in surfaces}

Let $E_1$ be an elliptic curve over a perfect field $k$ of
characteristic distinct from $3$. In what follows, we will
consider $\delta\in H^1(k,E_1[3])$ with $\Ob(\delta)=0$. A
possible way of constructing an abelian surface $A$ such that
$\delta$ is visible in $A$ starts by taking a suitable elliptic curve
$E_2/k$ together with a $k$-group scheme isomorphism $\lambda:
E_1[3]\to E_2[3]$. Let $\Delta\subset E_1\times E_2$ be the graph
of $-\lambda$ so that
$$\Delta(\kbar)=\{(P,-\lambda(P):P \in E_1[3](\kbar)\}.$$
Let $A:=(E_1\times E_2)/\Delta$ and write $\phi: E_1\times E_2\to
A$ for the corresponding isogeny. Since $\Delta\subset
E_1[3]\times E_2[3]$, we have another isogeny $\phi':A\to
E_1\times E_2$ such that $\phi'\circ \phi=3$. We write $p^*$ for
the composition $E_1\to (E_1\times E_2) \stackrel{\phi}{\to} A$
and $p_*$ for the composition $A\stackrel{\phi'}{\to}(E_1\times
E_2)\to E_1$ and $q^*, q_*$ for the corresponding morphisms
concerning $E_2$. It is straightforward to verify that $p^*,q^*$
are embeddings, that $\phi=p^*+q^*$ (where the projections are
understood and we note that the $+$ sign here corresponds to the
$-$ sign in the definition of $\Delta$) and that $\phi'=p_*\times
q_*$.

We combine the galois cohomology of the short exact sequences
\begin{align*}
&0\to E_1 \stackrel{p^*}{\to} A \stackrel{q_*}{\to} E_2 \to 0,\\
&0\to E_2 \stackrel{q^*}{\to} A \stackrel{p_*}{\to} E_1 \to 0, \text{ and}\\
&0\to E_i[3]\to E_i \stackrel{3}{\to}E_i\to 0\text{ for $i=1,2$}
\end{align*}
to obtain the big (symmetric) commutative diagram with exact rows and
columns
$$\xymatrix{
&&E_2(k)\ar[r]^{q^*}\ar[d]^{3}
  &A(k)\ar[d]^{q_*}\\
&&E_2(k)\ar@{=}[r]\ar[d]^{\alpha}
  &E_2(k)\ar[d]\\
E_1(k)\ar[r]^{3}\ar[d]_{p^*}
  &E_1(k)\ar[r]\ar@{=}[d]
  &H^1(k,\Delta)\ar[r]\ar[d]
  &H^1(k,E_1)\ar[d]\\
A(k)\ar[r]^{p_*}
  &E_1(k)\ar[r]
  &H^1(k,E_2)\ar[r]
  &H^1(k,A)
}$$ where we note that $H^1(k,\Delta)\simeq H^1(k,E_1[3])\simeq
H^1(k,E_2[3])$. We see that $\delta$ is visible in $A$ precisely
if $\delta\in H^1(k,E_1[3])=H^1(k,\Delta)$ lies in the image of
$\alpha$, i.e., if the curve $C_{\lambda(\delta)}$ corresponding
to $\lambda(\delta)\in H^1(k,E_2[3])$ has a rational point. We
summarize these observations, which are due to Mazur.

\begin{lemma}\label{lemma:mazurvis}
Let $E_1$ be an elliptic curve over a perfect field $k$ of
characteristic distinct from $3$ and let $\delta\in H^1(k,E[3])$
with $\Ob(\delta)=0$. Suppose that there exists an elliptic curve
$E_2/k$ and a $k$-group scheme isomorphism $\lambda: E_1[3]\to
E_2[3]$ such that the curve $C_{\lambda(\delta)}$ corresponding to
$\lambda(\delta)$ has a $k$-rational point. Then $\delta$ is
visible in the abelian surface $(E_1 \times E_2)/\Delta$ where
$\Delta$ denotes the graph of $-\lambda$.
\end{lemma}

Mazur also observed, in the case of a number field $k$, that if
$\delta\in S^{(3)}(E/k)$, then $C_\delta$ admits a plane cubic
model. Furthermore, there is a pencil of cubics through the $9$
flexes of $C_\delta$, and each non-singular member corresponds to
a $3$-covering $C_t\to E_t$, where $E_t[3]\simeq E[3]$ and $C_t\to
E_t$ represents $\delta$. It is therefore easy to find a $t$ such
that $C_t$ has a rational point; simply pick a rational point and
solve for $t$. To refine the construction, one can ask

\begin{question} Can one make $\delta\in H^1(k,E[3])$ visible in the jacobian of a genus $2$ curve?
\end{question}

Note that $E_1\times E_2$ is principally polarized via the product
polarization. This gives rise to a Weil pairing on $(E_1\times
E_2)[3]$, corresponding to the product pairing. If $A$ is a
jacobian, then $A$ must be principally polarized over $\kbar$. One way this
could happen is if the isogeny $p^*+q^*:E_1\times E_2\to A$ gives
rise to a principal polarization. This would be the case if the
kernel $\Delta$ is a maximal isotropic subgroup of $E_1[3]\times
E_2[3]$ with respect to the product pairing. That means that
$\lambda: E_1[3]\to E_2[3]$ must be an \emph{anti}-isometry, i.e.
for all $P,Q\in E_1[3]$ we must have
$$e_{E_2}(\lambda(P),\lambda(Q))=e_{E_1}(P,Q)^{-1}.$$
Note that the original cubic $C$ is a member of the pencil that
Mazur constructs, so in his construction $\lambda$ is actually an
\emph{isometry}, i.e. it preserves the Weil-pairing. Below we
consider a pencil of cubics that leads to an anti-isometry
$\lambda$.

\section{Anti-isometric pencils}\label{sec:antipencils}

Let $k$ be a perfect field of characteristic distinct from $2,3$.
Following \cite{fisher:hessian}, we associate to a ternary cubic
form $F \in k[x,y,z]$ three more ternary cubic forms. Namely, the
Hessian of $F$
\[
H(F):= -\frac{1}{2} \left| \begin{array}{ccc}
\frac{\partial{F}^2}{\partial{x}\partial{x}} &
\frac{\partial{F}^2}{\partial{x}\partial{y}} &
\frac{\partial{F}^2}{\partial{x}\partial{z}}\\
\frac{\partial{F}^2}{\partial{y}\partial{x}} &
\frac{\partial{F}^2}{\partial{y}\partial{y}} &
\frac{\partial{F}^2}{\partial{y}\partial{z}}\\
\frac{\partial{F}^2}{\partial{z}\partial{x}} &
\frac{\partial{F}^2}{\partial{z}\partial{y}} &
\frac{\partial{F}^2}{\partial{z}\partial{z}}
\end{array} \right|,
\]
the Caylean of $F$
\[
P(F):= -\frac{1}{x y z}\left| \begin{array}{ccc}
\frac{\partial{F}}{\partial{x}}(0,z,-y) & \frac{\partial{F}}{\partial{y}}(0,z,-y) & \frac{\partial{F}}{\partial{z}}(0,z,-y)\\
\frac{\partial{F}}{\partial{x}}(-z,0,x) & \frac{\partial{F}}{\partial{y}}(-z,0,x) & \frac{\partial{F}}{\partial{z}}(-z,0,x)\\
\frac{\partial{F}}{\partial{x}}(y,-x,0) &
\frac{\partial{F}}{\partial{y}}(y,-x,0) &
\frac{\partial{F}}{\partial{z}}(y,-x,0)
\end{array} \right|
\]
and a ternary cubic form denoted $Q(F)$, for which we refer to
\cite{fisher:hessian}*{Section~11.2}. For most cases one can take $Q(F)$ to be $H(P(F))$ or $P(H(F))$, but there are some exceptional cases where $P(F),Q(F)$ span an appropriate pencil and $P(F), H(P(F))$ do not.
The left action of $\GL_3$ on $k^3$ induces a right action of $\GL_3$ on ternary
cubic forms (or, more generally, on $k[x,y,z]$). For a ternary
cubic form $F$ and an $M \in \GL_3$ we denote this action simply by
$F \circ M$. The significance of the three associated ternary
cubic forms lies in the fact that $H(F)$ depends covariantly on
$F$ (of weight $2$) and $P(F)$ and $Q(F)$ depend contravariantly
on $F$ (of weights $4$ and $6$ respectively). This means that for
every ternary cubic form $F$ and every $M \in \GL_3$ we have, with
$d:=\det M$ that
\begin{eqnarray*}
H(F \circ M) & = & d^2 H(F) \circ M\\
P(F \circ M) & = & d^4 P(F) \circ M^{-T}\\
Q(F \circ M) & = & d^6 Q(F) \circ M^{-T},
\end{eqnarray*}
where $M^{-T}$ denotes the inverse transpose of $M$.

Now consider a smooth cubic curve $C$ in $\PP^2$ given by the zero
locus of a ternary cubic form $F$. Then $C$ has exactly $9$
different flex points $\Phi$, which all lie on the (not
necessarily smooth) curve given by $H(F)=0$. The smoothness of $C$
guarantees that $F$ and $H(F)$ will be linearly independent over
$k$. Hence $\Phi$ can be described as the intersection $F=H(F)=0$. We
call $\Phi$ the \emph{flex scheme} of $C$. At least one of $P(F)$
and $Q(F)$ turns out to be nonsingular (still assuming that $C$ is
nonsingular) and the intersection $P(F)=Q(F)=0$ equals the flex
points $\Phi^*$ of the nonsingular cubics among $P(F)$ and $Q(F)$
(if, say, $P(F)$ is nonsingular, then $\Phi^*$ can of course also
be written as $P(F)=H(P(F))=0$).

We can consider the pencil of cubics through $\Phi$,
explicitly given by
\[C_{(s:t)}: sF(x,y,z)+tH(F)(x,y,z)=0.\]
Classical invariant theory tells us the following. This pencil has exactly $4$ singular members and all other members
have flex scheme equal to $\Phi$. Conversely, any nonsingular
cubic with flex scheme $\Phi$ occurs in this pencil. Furthermore,
both $P(sF+tH(F))$ and $Q(sF+tH(F))$ are linear combinations of
$P(F)$ and $Q(F)$. This shows that the flex scheme $\Phi^*$ is
independent of the choice of $C$ through $\Phi$ and only depends
on $\Phi$. We call $\Phi^*$ the dual flex scheme of $\Phi$ and we will
justify this name below.

\begin{remark}
In the discussion above it was convenient to consider just one
projective plane $\PP^2$. A more canonical way would be to
consider a projective plane $\PP^2$ with coordinates $x,y,z$
and the dual projective plane, denoted $(\PP^2)^*$, whose
coordinates $u,v,w$ are related to those of $\PP^2$ by $ux+vy+wz=0$.
Now let $C$ be a smooth cubic curve in
$\PP^2$ given by the zero locus of the ternary cubic form
$F(x,y,z)$ with flex scheme $\Phi$. The $9$ tangent lines
through $\Phi$ determine $9$ points in $(\PP^2)^*$. Generically,
these $9$ points in $(\PP^2)^*$ will not be the flex points of a
smooth cubic curve, hence generically there will a unique cubic
curve going through these points. This curve in $(\PP^2)^*$ is
exactly given by the zero locus of the Caylean, i.e.
$P(F)(u,v,w)=0$; see also \cite{salmon}*{pp.151,190--191}.
Moreover, if the characteristic of $k$ is zero, then it turns out
that this cubic curve is nonsingular if and only if the
$j$-invariant of $C$ is nonzero.
\end{remark}

As a simple, but important example we take $F:=x^3+y^3+z^3$. Then
we compute
\[H(F)=-108xyz, \quad P(F)=-54xyz, \quad Q(F)=324(x^3+y^3+z^3).\]
Now define $\Phi_0$ to be the flex scheme of $F=0$, i.e.
\begin{equation}\label{eq:flex scheme 0}
\Phi_0:=\{[x:y:z] \in \PP^2: x^3+y^3+z^3=xyz=0\}.
\end{equation}
Then we see that the flex scheme given by $P(F)=Q(F)=0$ (which is
the flex scheme of $Q(F)=0$) equals $\Phi_0$, i.e.
\[\Phi_0^*=\Phi_0.\]

Geometrically all flex schemes are linear transformations of each
other. In particular, for any flex scheme $\Phi$ there exists an
$M \in \GL_3(\kbar)$ such that $\Phi=M\Phi_0$. The contravariance
of $P$ and $Q$ implies that the assignment $\Phi \mapsto \Phi^*$
has the contravariance property that for any flex scheme $\Phi$
and $M \in \GL_3$
\begin{equation}\label{eq:contravar flex scheme}
(M \Phi)^*=M^{-T} \Phi^*.
\end{equation}
We also note that this implies that the assignment $\Phi \mapsto
\Phi^{**}:=(\Phi^*)^*$ is covariant in the sense that for any flex
scheme $\Phi$ and $M \in \GL_3$ we have $(M \Phi)^{**}=M
\Phi^{**}.$ Writing $\Phi=M\Phi_0$ and using
$(\Phi_0)^{**}=\Phi_0^*=\Phi_0$ we now get
\[\Phi^{**}=(M \Phi_0)^{**}=M \Phi_0^{**}=M \Phi_0=\Phi.\]
This justifies calling $\Phi^*$ the \emph{dual} flex scheme of
$\Phi$.


To any flex scheme $\Phi$ we associate a group $\Theta(\Phi)
\subset \GL_3$ as follows. Choose a nonsingular cubic curve $C$
through $\Phi$ and let $E$ be its jacobian. After identifying $E$
and $C$ as curves over $\kbar$, we get an action of $E[3]$ on $C$,
which extends to a linear action on $\PP^2$. This determines an
embedding $\chi:E[3] \to \PGL_3$. Obviously, the image
$\chi(E[3])$ only depends on $\Phi$. We define $\Theta(\Phi)$ to
be the inverse image of $\chi(E[3])$ in $\GL_3$. Actually
$\Theta(\Phi)$ can be defined just in terms of $\Phi$, without
choosing $C$, since it turns out that $\chi(E[3])$ consists
exactly of the linear transformations that preserve $\Phi$. (One
way of quickly finding these linear transformations explicitly is
by using the fact that, for any two distinct points of $\Phi$, the
line through these two points intersects $\Phi$ in a unique third
point.) The construction gives rise to the theta group
\[1 \to \GG_m \to \Theta(\Phi) \to E[3] \to 1.\]
Note that the isomorphism class of this theta group may still
depend on the choice of identification of $C$ with $E$. This
corresponds to the choice of an isomorphism between
$\Theta(\Phi)/\GG_m $ and $E[3]$. If $\Phi$ is defined over $k$,
then $E[3]$ and $\Theta(\Phi)$ are also defined over $k$ and the
element in $H^1(k,E[3])$ corresponding to this theta group is the
same as the element corresponding to the $3$-covering $C \to
C/E[3] \simeq E$ for any nonsingular cubic curve $C$ through
$\Phi$. The construction also shows that for any $M \in \GL_3$ we
have
\begin{equation}\label{eq:theta change of co}
\Theta(M\Phi)=M\Theta(\Phi)M^{-1}.
\end{equation}

\begin{prop}\label{prop:anti-theta}
Let $\Phi_1 \subset \PP^2$ be a flex scheme and let
$\Phi_2:=\Phi_1^*$ be the dual flex scheme. For $i=1,2$ let $C_i$
be a smooth plane cubic with flex scheme $\Phi_i$, denote its
jacobian by $E_i$ and consider an induced theta group
\begin{equation}\label{eq:theta of flex}
\xymatrix{ 1\ar[r]&
  \GG_m\ar[r]^{\alpha_i}&
  \Theta(\Phi_i)\ar[r]^{\beta_i} &
  E_i[3]\ar[r]& 1 }.
\end{equation}
Then the outer automorphism $(-T):\GL_3 \to \GL_3$ given by $M
\mapsto M^{-T}$, yields an
isomorphism $\Theta(\Phi_1) \to \Theta(\Phi_2)$.
There exists an anti-isometry $\lambda: E_1[3] \to
E_2[3]$ making the following diagram commutative.
\begin{equation}\label{eq:theta anti-theta}
\xymatrix{ 1\ar[r]&
  \GG_m\ar[r]^{\alpha_1}\ar[d]_{x\mapsto x^{-1}}&
  \Theta(\Phi_1)\ar[r]^{\beta_1}\ar[d]^{(-T)}&
  E_1[3]\ar[r]\ar[d]^{\lambda}&
  1\\
1\ar[r]&
  \GG_m\ar[r]^{\alpha_2}&
  \Theta(\Phi_2)\ar[r]^{\beta_2}&
  E_2[3]\ar[r]&
  1
}
\end{equation}
In particular, let $\delta_i \in H^1(k,E_i[3])$ correspond to the
theta group (\ref{eq:theta of flex}). Then under the isomorphism
$H^1(k,E_1[3]) \simeq H^1(k,E_2[3])$ induced by $\lambda$, the cocycle
$\delta_1$ maps to $\delta_2$.
\end{prop}

\begin{proof}
Once the isomorphism $\Theta(\Phi_1) \to \Theta(\Phi_2)$ given by $ M
\mapsto M^{-T}$ is established, the existence of an isomorphism
$\lambda: E_1[3] \to E_2[3]$ making the diagram (\ref{eq:theta
anti-theta}) commutative, follows immediately. That $\lambda$ must be an anti-isometry can readily be seen as
follows. Let $P,Q \in E_1[3]$ and choose $x,y \in \Theta(\Phi_1)$
such that $P=\beta_1(x)$ and $Q=\beta_1(y)$. Then
\begin{eqnarray*}
\alpha_2(e_{E_2}(\lambda(P),\lambda(Q))) & = &
\alpha_2(e_{E_2}(\beta_2(x^{-T}),\beta_2(y^{-T})))\\
 & = & x^{-T}y^{-T} x^T y^T\\
 & = & (xyx^{-1}y^{-1})^{-T}\\
 & = & \alpha_1(e_{E_1}(\beta_1(x),\beta_1(y)))^{-T}\\
 & = & \alpha_1(e_{E_1}(P,Q)^{-1}).
\end{eqnarray*}
The last statement of the proposition is also immediate, so we are
left with establishing $(-T):\Theta(\Phi_1) \stackrel{\sim}{\to}
\Theta(\Phi_2)$. It suffices to show that for a flex scheme $\Phi
\subset \PP^2$ we have $\Theta(\Phi)^{-T}=\Theta(\Phi^*)$. Write
$\Phi=M\Phi_0$ for some $M \in \GL_3$ with $\Phi_0$ given by
(\ref{eq:flex scheme 0}). Then a straightforward calculation shows
that $\Theta(\Phi_0)^{-T}=\Theta(\Phi_0)$. We also know that
$\Phi_0^*=\Phi_0$, so we get
$\Theta(\Phi_0)^{-T}=\Theta(\Phi_0^*)$. Together with
(\ref{eq:contravar flex scheme}) and (\ref{eq:theta change of co})
we finally obtain,
\begin{eqnarray*}
\Theta(\Phi)^{-T} & = & \Theta(M\Phi_0)^{-T}\\
 & = & M^{-T}\Theta(\Phi_0)^{-T} M^T\\
 & = & M^{-T} \Theta(\Phi_0^*)(M^{-T})^{-1}\\
 & = & \Theta(M^{-T}\Phi_0^*)\\
 & = & \Theta((M\Phi_0)^*)\\
 & = & \Theta(\Phi^*).
\end{eqnarray*}
\end{proof}

\begin{remark}
The construction above of the dual flex scheme $\Phi^*$ of a flex scheme
$\Phi$ involved choosing a smooth cubic going through
$\Phi$. Without using theta groups, it was not obvious from this
construction that the degree $9$ \'etale algebra $k(\Phi)$ is
isomorphic to $k(\Phi^*)$. However, there exists a nice explicit
geometric construction of the dual flex scheme that remedies these
shortcomings of the earlier construction. Given a flex scheme
$\Phi$, we proceed as follows. We label its $9$ points over
$\kbar$ with $P_1,\ldots,P_9$. There are $4$ sets of $3$ lines,
(corresponding to the $4$ singular member of the pencil of cubics
through $\phi$) containing these points. We label the line that
contains $P_i,P_j,P_k$ with $l_{\{i,j,k\}}$. One can label the
points such that the subscripts are
$$\begin{array}{c}
   \{1,2,3\}\\
   \{4,5,6\}\\
   \{7,8,9\}
\end{array},
\begin{array}{c}
   \{1,4,7\}\\
   \{2,5,8\}\\
   \{3,6,9\}
\end{array},
\begin{array}{c}
   \{1,5,9\}\\
   \{2,6,7\}\\
   \{3,4,8\}
\end{array},
\begin{array}{c}
   \{1,6,8\}\\
   \{2,4,9\}\\
   \{3,5,7\}
\end{array},
$$
Naturally, two different lines $l_{\{i_1,j_1,k_1\}},
l_{\{i_2,j_2,k_2\}}$ meet in a unique point. If for example
$i_1=i_2$, then the intersection point is $P_{i_1}$. If the two
sets $\{i_1,j_1,k_1\}$ $\{i_2,j_2,k_2\}$ are disjoint, then the
two lines meet in a point outside $\Phi$. We name this point
$L_{\{i_3,j_3,k_3\}}$, where
$\{i_1,j_1,k_1,i_2,j_2,k_2,i_3,j_3,k_3\}=\{1,\ldots,9\}$. As it
turns out, the four points that have $i$ in their label all lie on
a line $p_i$. It is also straightforward to check that the $p_i$
together with the $L_{\{i,j,k\}}$ form a configuration in
$(\PP^2)^*$ that is completely dual to the $P_i$ with the
$l_{\{i,j,k\}}$. The $p_i$ form the $\kbar$ points of a flex
scheme in $(\PP^2)^*$, which is justifiably a flex scheme $\Phi^*$
dual to $\Phi$, and its construction immediately implies the
contravariance property $(M\Phi)^*=M^{-T}\Phi^*$.

We can easily verify that the two constructions of $\Phi^*$
coincide for one flex scheme, for instance $\Phi_0$. The general
result then follows because any flex scheme can be expressed as
$M\Phi_0$ for some $M\in\GL_3(\kbar)$.

Since the action of $\Gal(\kbar/k)$ on $\{P_1,\ldots,P_9\}$ must
act via collinearity-preserving permutations, we see that if
$\sigma(P_i)=P_{\sigma(i)}$ then $\sigma(p_i)=p_{\sigma(i)}$.
Hence, we see that the $\kbar$-points of $\Phi$ and its dual have
the same Galois action and hence $k(\Phi)$ is isomorphic as a
$k$-algebra to $k(\Phi^*)$.

\end{remark}

\section{Recovering the genus $2$ curve}\label{sec:genus2}

Let $k$ be a field and let $E_1,E_2$ be two elliptic curves over $k$ with
an anti-isometry
$\lambda: E_1[3] \to E_2[3]$ and denote by
$\Delta$ the graph of $-\lambda$ as before. Recall that $E_1
\times E_2$ is principally polarized via the product polarization
and that the induced polarization on $A:=(E_1 \times E_2)/\Delta$ is
also principal in this case. It is a classical fact that if $A$ is
not geometrically isomorphic to a product of elliptic curves, then
$A$ (together with its principal polarization) is
isomorphic to the jacobian of a genus $2$ curve $C$. Let us assume
from now on that $E_1$ and $E_2$ are non-isogenous. In
\cite{frey-kani} it is shown that in this case $A$ is always
isomorphic over $k$ to the jacobian of a genus $2$ curve $C/k$.
This is enough to get our main theoretical result.
\begin{theorem}
Let $E$ be an elliptic curve over a number field $k$ and let $\xi
\in \Sha(E/k)[3]$. Then $\xi$ is visible in the jacobian of a
genus $2$ curve $C/k$.
\end{theorem}

\begin{proof}
Let $\delta\in S^{(3)}(E/k)$ be a cocycle representing $\xi$. By
Proposition~\ref{prop:H1theta}, there is a $3$-covering
$C_\delta\to E$ corresponding to $\delta$. According to
Remark~\ref{rem:brauerob}, we have that $\Ob(\delta)=0$ and hence
that $C_\delta\subset\PP^2$. Let $\Phi\subset\PP^2$ be its flex
scheme. The construction in Section~\ref{sec:antipencils} gives us
a pencil of cubics through $\Phi^*$, so we can easily pick a
non-singular one with a rational point. It follows from
Proposition~\ref{prop:anti-theta} that such a curve is of the form
$C_{\lambda(\delta)}$ for some elliptic curve $E_2$ and some
anti-isometry $\lambda: E[3]\to E_2[3]$.

This places us in the situation of Lemma~\ref{lemma:mazurvis}, so $\delta$ is visible in an abelian surface $A=(E\times E_2)/\Delta$. We have ensured that $\lambda$ is an anti-isometry, which implies that the surface is principally polarized. As long as we make sure that $E,E_2$ are non-isogenous (and this is easy given the freedom we have in choosing $C_\lambda(\delta)$)
it follows that $A$ is a jacobian.
\end{proof}

\begin{remark}
We could of course state a more general result about visibility of
elements $\delta \in H^1(k,E[3])$ with $\Ob(\delta)=0$ for an
elliptic curves $E$ over a perfect field $k$ of characteristic
distinct from $2$ or $3$. Note however that if $k$ is too small,
there might not be enough non-isogenous elliptic curves available.
The exclusion of fields of characteristic $3$ is a serious one,
the exclusion of non-perfect fields less so. Most of what we are
saying could be generalized to the non-perfect case, basically
because for an elliptic curve over any field of characteristic
distinct from $3$, the multiplication by $3$ map is separable. The
exclusion of fields of characteristic $2$ stems from the fact that
the necessary invariant theory in this case is not readily
available.
\end{remark}

We continue with the construction of the genus $2$ curve $C$.
Define the divisor $\Theta:=0_1 \times E_2+E_1 \times 0_2$ on $E_1
\times E_2$, which gives a principal polarization on $E_1 \times
E_2$. Next, consider the set $\sD$ of effective divisors on
$E_1 \times E_2$ over $\kbar$ which are linear equivalent to $3
\Theta$ and invariant under $\Delta$. Also consider the set
$\sC$ of effective divisors $C$ on $A$ over $\kbar$ whose
pull-back to $E_1 \times E_2$ are linear equivalent to $3 \Theta$
and which satisfy $(C \cdot C)=2$. Frey and Kani show that there
exist unique curves $D \in \sD$ and $C \in \sC$
defined over $k$ which are invariant under multiplication by $-1$.
Furthermore, because $E_1$ and $E_2$ are not isogenous, $D$ and
$C$ are irreducible smooth curves of genus $10$ and $2$
respectively and the natural map $D \to C$ is unramified of degree
$9$.

If $k$ is a
perfect field of characteristic distinct from $2$ or $3$, the
curves $D$ and $C$ can be explicitly constructed as follows. Embed
$E_1$ in $\PP^2$, given by, say $F(x,y,z)=0$, for a ternary cubic
$F/k$ (such an $F$ is readily obtained if $E_1$ is given by a
Weierstrass model). Express $E_2$ as $G:=s P(F)+tQ(F)=0$ for some
$s,t \in k$. This way, we obtain an embedding of $E_1 \times E_2$
in $\PP^2 \times \PP^2$ given by
\[F(x,y,z)=G(u,v,w)=0.\]
Moreover, by appealing to Proposition \ref{prop:anti-theta} we
obtain that the curve on this surface given by $xu+yv+zw=0$ must
be the curve $D$. The genus $2$ curve $C$ is the image of $D$ in $(E_1\times E_2)/\Delta$.
$$\xymatrix{
E_1\times E_2\ar[dd]_{[3]\times[3]}\ar[dr]\\
&(E_1\times E_2)/\Delta\ar[dl]\\
E_1\times E_2}$$
The map $[3]\times[3]$ is much more accessible, though. Also observe that
the hypersurface $\{xu+yv+zw=0\}\subset \PP^2\times\PP^2$ is only invariant under $\Delta\subset E_1[3]\times E_2[3]$. A little extra work shows that the subgroup of $E_1[3]\times E_2[3]$ under which $D$ is invariant is equal to $\Delta$. Hence, we can find a model of $C$ as a curve on $E_1\times E_2$ by computing $([3]\times[3])(D)$. This can easily be done via interpolation, as explained below by means of an example.

\section{Examples}

Following the first example in \cite{cremaz:vis}*{Table~1},
consider the elliptic curve $681b1$ (in Cremona's notation), given
by the minimal Weierstrass equation
\[E_1: y^2 + xy = x^3 + x^2 - 1154x - 15345.\]
It turns out that the plane cubic curve
\[C_1: x^3+5x^2y+5x^2z+2xy^2+xyz+xz^2+y^3-5y^2z+2yz^2+6z^3=0\]
defines an element $\xi$ (up to inverse) of order three in
$\Sha(E_1/\QQ)$. The contravariants of the cubic above defining
$C_1$, denoted $P_0, Q_0$, are given by
\begin{eqnarray*}
P_0 & = & -478x^3 + 2525x^2y + 916x^2z - 1127xy^2 + 29xyz \\
 & & -160xz^2 + 753y^3 - 1228y^2z + 260yz^2 + 301z^3,\\
Q_0 & = & -122314x^3 + 618551x^2y + 191092x^2z - 271157xy^2 - 7825xyz \\
 & & - 28120xz^2 + 184011y^3 - 264916y^2z + 55892yz^2 + 73663z^3.
\end{eqnarray*}
Now the curve
\[C_2: 55033P_0-235Q_0=0\]
has a rational point $[x:y:z]=[10:8:7]$ and its jacobian is the
elliptic curve $681c1$, given by the minimal Weierstrass equation
\[E_2: y^2 + y = x^3 - x^2 + 2.\]

To construct the corresponding genus two curve $C$ such that $\xi$
becomes visible in its jacobian we could now take the curve in
$C_1\times C_2\subset \PP^2\times\PP^2$ with coordinates
$([x:y:z],[u:v:w])$ given by the equation $xu+yv+zw=0$, and take
its image under $C_1\times C_2\to E_1\times E_2$, since this is a
twist of $[3]\times [3]: E_1\times E_2\to E_1\times E_2$ anyway.
We will follow Section~\ref{sec:genus2} more closely. Obviously,
$E_1$ is given by $F=0$ if we define
\[F:= y^2z + xyz -( x^3 + x^2z - 1154xz^2 - 15345z^3).\]
The contravariants of the ternary cubic $F$ are given by
\begin{eqnarray*}
P & = & -2308x^3 + 3462x^2y - 5x^2z - 275056xy^2 + 5xyz\\
 & & + 6xz^2 + 136951y^3 + 13853y^2z - 3yz^2,\\
Q & = & -725020x^3 + 1087530x^2y + 27721x^2z - 65861608xy^2 - 27721xyz\\
 & & - 30xz^2 + 32749549y^3 + 3217559y^2z + 15yz^2 + 24z^3.\\
\end{eqnarray*}
Write $j(s,t)$ for the $j$-invariant of the curve given by
$sP+tQ=0$. The $j$-invariant of $E_2$ equals $-4096/2043$ and the
equation $j(s,t)=-4096/2043$ has exactly one solution in
$\PP^1(\QQ)$, namely $[s:t]=[55033:-235]$ (compare with the
definition of $C_2$). This gives us a new model for $E_2$, namely
\[ E_2: 55033P-235Q=0.\]
We consider the surface $E_1 \times E_2$ embedded in $\PP^2 \times \PP^2$  as
\[F(x,y,z)=0, \quad 55033P(u,v,w)-235Q(u,v,w)=0.\]
Now $D$ is simply the curve on this surface given by
\[xu+yv+zw=0.\]
The image of $D$ under multiplication by $3$ on $E_1 \times E_2$
is the genus two curve $C$. Using the defining properties of $C$
from Section~\ref{sec:genus2} (such as the invariance under
multiplication by $-1$), we get that as a curve on $E_1 \times
E_2$ it must be of the form
\[axu+byv+czw+dxw+ezu=0\]
for some $a,b,c,d,e \in \QQ$.
Now we simply generate $4$ points on $C$ (over a number field), compute the image under multiplication by $3$ of these points and solve for $a,b,c,d,e$. If the dimension of the solution space is greater than $1$, we must of course add points (or take $4$ better ones) so that the solution space becomes $1-$dimensional.
This gives us our equation for $C$. By a linear change of the $u,v,w$ coordinates we can change the model for $E_2$ back to the original minimal Weierstrass model. Thus, the model for $E_1 \times E_2$ embedded in $\PP^2 \times \PP^2$ is
\[E_1: y^2z + xyz = x^3 + x^2z - 1154xz^2 - 15345z^3,\]
\[E_2: v^2w + vw^2 = u^3 - u^2w + 2w^3\]
and $C$ is the curve on this surface given by
\[4xu - 155zu + xv + 2yv - 40xw + yw + 1314zw=0.\]
Hyperelliptic models for $C$ are
\begin{eqnarray*}
Y^2 + (X + 1)Y & = & 3X^5 + 5X^4 + X^3 - 8X^2 - 5X+ 2\text{ or }\\
Y^2 & = & (3X-1)(X+1)(4X^3+4X^2-9).
\end{eqnarray*}

Next, consider the elliptic curve $2006e1$, given by the minimal
Weierstrass equation
\[E_1: y^2 + xy = x^3 + x^2 - 58293654x - 171333232940.\]
It turns out that the plane cubic curve
\[C_1: 20x^3 + 44x^2y + 21x^2z - 77xy^2 + 71xyz + 44xz^2 + 31y^3 + 3y^2z
+ 150yz^2 + z^3=0\] defines an element $\xi$ (up to inverse) of
order three in $\Sha(E_1/\QQ)$. In the sixth example in
\cite{cremaz:vis}*{Table~1} the elliptic curve $E_2$ which \lq
explains\rq\ $\Sha(E_1/\QQ)$ is $2006d1$. However, for this choice
of $E_2$, there only exists an isometry between $E_1[3]$ and
$E_2[3]$ and not an anti-isometry. The corresponding abelian
surface $(E_1 \times E_2)/\Delta$ visualizing $\xi$ will not be
the jacobian of a genus $2$ curve. If instead we take for $E_2$
the elliptic curve $6018c1$, then we do have an anti-isometry
between $E_1[3]$ and $E_2[3]$. Following the same route as in the
first example, we find that $\xi$ is visible in the jacobian of
the genus $2$ curve $C$ with hyperelliptic models
\begin{eqnarray*}
Y^2 + (X^2 + X)Y & = & -9675X^6 - 94041X^5 - 914X^4 + 1301674X^3 - 352310X^2\\ & & - 2071181X - 945269\text{ or }\\
Y^2 & = & 43(2X + 13)(18X^2 - 81X + 89)(25X^3 + 193X^2 + 224X +
76).
\end{eqnarray*}

\end{document}